\documentclass[12pt]{amsart}
\usepackage{amsmath,amsthm,amsfonts,amssymb,mathrsfs}
\date{\today}

\usepackage{color}
\input xymatrix
\xyoption{all}

\usepackage[colorlinks,plainpages,citecolor=magenta, linkcolor=blue, backref]{hyperref}%,bookmarksnumbered,

\usepackage{hyperref}

\newtheorem{theorem}{Theorem}%[section]

\newtheorem{proposition}{Proposition}

\newtheorem{lemma}{Lemma}
\theoremstyle{definition}
\newtheorem{example}{Example}%[section]
\newtheorem{remark}{Remark}%[section]

\newtheorem{definition}{Definition}%[section]

\begin{document}

\title[On non-topologizable semigroups]{On non-topologizable semigroups}

\author{Oleg~Gutik}
\address{Faculty of Mathematics, National University of Lviv,
Universytetska 1, Lviv, 79000, Ukraine}
\email{oleg.gutik@lnu.edu.ua,
ovgutik@yahoo.com}



\begin{abstract}
We find anti-isomorphic submonoids $\mathscr{C}_{+}(a,b)$  and $\mathscr{C}_{-}(a,b)$ of the bicyclic monoid $\mathscr{C}(a,b)$ with the following properties:
every Hausdorff left-continuous (right-continuous) topology on  $\mathscr{C}_{+}(a,b)$  ($\mathscr{C}_{-}(a,b)$) is discrete and  there exists a compact Hausdorff topological monoid $S$ which contains $\mathscr{C}_{+}(a,b)$  ($\mathscr{C}_{-}(a,b)$) as a submonoid. Also, we construct a non-discrete right-continuous  (left-continuous)  topology $\tau_p^+$ ($\tau_p^-$) on the semigroup $\mathscr{C}_{+}(a,b)$ ($\mathscr{C}_{-}(a,b)$) which is not left-continuous (right-continuous).
\end{abstract}

\maketitle

\section{\textbf{Introduction, motivation and main definitions}}\label{section-1}

In this paper we shall follow the terminology of \cite{Carruth-Hildebrant-Koch=1983, Clifford-Preston=1961,  Engelking=1989, Ruppert=1984}.

\smallskip

By $\omega$ we denote the set of all non-negative integers. Throughout these notes we always assume that all topological spaces involved are Hausdorff --- unless explicitly stated otherwise.

\begin{definition}
Let $X$, $Y$ and $Z$ be topological spaces. A map $f\colon X\times Y\to Z$, $(x,y)\mapsto f(x,y)$, is called
\begin{itemize}
  \item[$(i)$] \emph{right} [\emph{left}] \emph{continuous} if it is continuous in the right [left] variable; i.e., for every fixed $x_0\in X$ [$y_0\in Y$] the map $Y\to Z$, $y\mapsto f(x_0,y)$ [$X\to Z$, $x\mapsto f(x,y_0)$] is continuous;
  \item[$(ii)$] \emph{separately continuous} if it is both left and right continuous;
  \item[$(iii)$] \emph{jointly continuous} if it is continuous as a map between the product space $X\times Y$ and the space $Z$.
\end{itemize}
\end{definition}

\begin{definition}[\cite{Carruth-Hildebrant-Koch=1983, Ruppert=1984}]
Let $S$ be a non-void topological space which is provided with an associative multiplication (a semigroup operation) $\mu\colon S\times S\to S$, $(x,y)\mapsto \mu(x,y)=xy$. Then the pair $(S,\mu)$ is called
\begin{itemize}
  \item[$(i)$] a \emph{right topological semigroup} if the map $\mu$ is right continuous, i.e., all interior left shifts $\lambda_s\colon S\to S$, $x\mapsto sx$, are continuous maps, $s\in S$;
  \item[$(ii)$] a \emph{left topological semigroup} if the map $\mu$ is left continuous, i.e., all interior right shifts $\rho_s\colon S\to S$, $x\mapsto xs$, are continuous maps, $s\in S$;
  \item[$(iii)$] a \emph{semitopological semigroup} if the map $\mu$ is separately continuous;
  \item[$(iv)$] a \emph{topological semigroup} if the map $\mu$ is jointly continuous.
\end{itemize}

We usually omit the reference to $\mu$ and write simply $S$ instead of $(S,\mu)$. It goes without saying that every topological semigroup is
also semitopological and every semitopological semigroup is both a right and left topological semigroup.
\end{definition}

A topology $\tau$ on a semigroup $S$ is called:
\begin{itemize}
  \item a \emph{semigroup} topology if $(S,\tau)$ is a topological semigroup;
  %\item an \emph{inverse semigroup} topology if $(S,\tau)$ is a topological inverse semigroup;
  \item a \emph{shift-continuous} topology if $(S,\tau)$ is a semitopological semigroup;
  \item an \emph{left-continuous} topology if $(S,\tau)$ is a left topological semigroup;
  \item an \emph{right-continuous} topology if $(S,\tau)$ is a right topological semigroup.
\end{itemize}

The \emph{bicyclic monoid} ${\mathscr{C}}(a,b)$ is the semigroup with the identity $1$ generated by two elements $a$ and $b$ subjected only to the condition $ab=1$. The semigroup operation on ${\mathscr{C}}(a,b)$ is determined as
follows:
\begin{equation*}\label{eq-1}
    b^ka^l\cdot b^ma^n=
    \left\{
      \begin{array}{ll}
        b^{k-l+m}a^n, & \hbox{if~} l<m;\\
        b^ka^n,       & \hbox{if~} l=m;\\
        b^ka^{l-m+n}, & \hbox{if~} l>m.
      \end{array}
    \right.
\end{equation*}
It is well known that the bicyclic monoid ${\mathscr{C}}(a,b)$ is a bisimple (and hence simple) combinatorial $E$-unitary inverse semigroup and every non-trivial congruence on ${\mathscr{C}}(a,b)$ is a group congruence \cite{Clifford-Preston=1961}.

\smallskip

For a semigroups $S$ and $T$ a map $\alpha\colon S\to S$ is said to be an \emph{anti-homomorphism} if $\alpha(s\cdot t)=\alpha(t)\cdot \alpha(s)$. A bijective anti-homomorphism of semigroups is called an \emph{anti-isomorphism}.

\smallskip

It is well known that topological algebra studies the influence of topological properties of its objects on their algebraic properties and the influence of algebraic properties of its objects on their topological properties. There are two main problems in topological algebra: the problem of non-discrete topologization and the problem of embedding into objects with some topological-algebraic properties.

\smallskip

In mathematical literature the question about non-discrete (Hausdorff) topologization was posed by Markov in \cite{Markov=1945}. Pontryagin gave well known conditions on a base at the unity of a group for its non-discrete topologization (see Theorem~4.5 of \cite{Hewitt-Roos=1963}). Various authors have refined Markov's question: \emph{can a given infinite group $G$ endowed with a non-discrete group topology be embedded into a compact topological group?} Again, for an arbitrary Abelian group $G$ the answer is affirmative, but there is a non-Abelian topological group that cannot be embedded into any compact
topological group ({see Section~9 of \cite{HBSTT-1984}}).

\smallskip

Also, Ol'shanskiy \cite{Olshansky=1980} constructed an infinite countable group $G$ such that every Hausdorff group topology on $G$ is discrete. Taimanov presented in \cite{Taimanov=1973} a commutative semigroup $\mathfrak{T}$ which admits only discrete Hausdorff semigroup topology. Also in \cite{Taimanov=1975} he gave sufficient conditions on a commutative semigroup to have a non-discrete semigroup topology. In \cite{Gutik=2016} it is proved that each shift-continuous  $T_1$-topology on the Taimanov semigroup $\mathfrak{T}$ is discrete.

\smallskip

The bicyclic monoid admits only the discrete shift-continuous Hausdorff topology \cite{Eberhart-Selden=1969, Bertman-West=1976}. If a Hausdorff (semi)topological semigroup $T$ contains the bicyclic monoid ${\mathscr{C}}(a,b)$ as a dense proper semigroup then $T\setminus {\mathscr{C}}(a,b)$ is a closed ideal of $T$ \cite{Eberhart-Selden=1969, Gutik=2015}. Moreover, the closure of ${\mathscr{C}}(a,b)$ in a locally compact topological inverse semigroup can be obtained (up to isomorphism) from ${\mathscr{C}}(a,b)$ by adjoining the additive group of integers in a suitable way \cite{Eberhart-Selden=1969}.

\smallskip

Stable and $\Gamma$-compact topological semigroups do not contain the bicyclic monoid~\cite{Anderson-Hunter-Koch=1965, Hildebrant-Koch=1986, Koch-Wallace=1957}. The problem of embedding the bicyclic monoid into compact-like topological semigroups was studied in \cite{Banakh-Dimitrova-Gutik=2009, Banakh-Dimitrova-Gutik=2010, Bardyla-Ravsky=2020, Gutik-Repovs=2007}.

\smallskip

Subsemigroups of the bicyclic monoid are studied in \cite{Descalco-Ruskuc=2005, Descalco-Ruskuc=2008, Hovsepyan=2020}.

\smallskip

We define the following subsets of the bicyclic monoid
\begin{equation*}
  \mathscr{C}_{+}(a,b)=\left\{b^ia^j\in\mathscr{C}(a,b)\colon i\leqslant j\right\}
  \quad \hbox{and} \quad
  \mathscr{C}_{-}(a,b)=\left\{b^ia^j\in\mathscr{C}(a,b)\colon i\geqslant j\right\}.
\end{equation*}

\begin{proposition}\label{proposition-3}
$\mathscr{C}_{+}(a,b)$ and $\mathscr{C}_{-}(a,b)$  are submonoids of $\mathscr{C}(a,b)$.
\end{proposition}

\begin{proof}
For arbitrary $b^{i_1}a^{j_1}, b^{i_2}a^{j_2}\in \mathscr{C}_{+}(a,b)$ by the semigroup operation of the bicyclic monoid $\mathscr{C}(a,b)$ we have that
\begin{equation*}
  b^{i_1}a^{j_1}\cdot b^{i_2}a^{j_2}=
  \left\{
    \begin{array}{ll}
      b^{i_1-j_1+i_2}a^{j_2}, & \hbox{if~} j_1\leqslant i_2;\\
      b^{i_1}a^{j_1-i_2+j_2}, & \hbox{if~} j_1> i_2.
    \end{array}
  \right.
\end{equation*}

If $j_1\leqslant i_2$ we have that $i_1-j_1+i_2\geqslant i_2\geqslant 0$. Since $j_1\leqslant j_1$ and $j_2\leqslant j_2$ we get that $i_1-j_1+i_2\leqslant i_2$, and hence $i_1-j_1+i_2\leqslant j_2$.

\smallskip

If $j_1>i_2$ we get that $j_1-i_2+j_2\geqslant j_1$ because $i_2\leqslant j_2$. Hence $j_1-i_2+j_2\geqslant i_1$.

\smallskip

It is obvious that $1=b^0a^0$ is the identity element of $\mathscr{C}_{+}(a,b)$. This and above arguments imply that $\mathscr{C}_{+}(a,b)$ is a submonois of $\mathscr{C}(a,b)$.

\smallskip

The proof the statement that $\mathscr{C}_{-}(a,b)$ is a submonoid of $\mathscr{C}(a,b)$ is similar.
\end{proof}

In this paper we prove that every Hausdorff left-continuous (right-continuous) topology on the monoid $\mathscr{C}_{+}(a,b)$  ($\mathscr{C}_{-}(a,b)$) is discrete and show that there exists a compact Hausdorff topological monoid $S$ which contains $\mathscr{C}_{+}(a,b)$  ($\mathscr{C}_{-}(a,b)$) as a submonoid. Also, we construct a non-discrete right-continuous (left-continuous) topology $\tau_p^+$ ($\tau_p^-$) on the semigroup $\mathscr{C}_{+}(a,b)$ ($\mathscr{C}_{-}(a,b)$) which is not left-continuous (right-continuous).

\section{\textbf{Algebraic properties and topologizations of monoids $\mathscr{C}_{+}(a,b)$ and $\mathscr{C}_{-}(a,b)$}}

\begin{proposition}\label{proposition-4}
The monoids $\mathscr{C}_{+}(a,b)$ and $\mathscr{C}_{-}(a,b)$  are anti-isomorphic.
\end{proposition}

\begin{proof}
We define a map $\alpha\colon \mathscr{C}_{+}(a,b)\rightarrow\mathscr{C}_{-}(a,b)$ by the formula $\alpha(b^ia^j)=b^ja^i$. Then for any $b^{i_1}a^{j_1}$, $b^{i_2}a^{j_2}\in \mathscr{C}_{+}(a,b)$ we have that
\begin{align*}
  \alpha(b^{i_1}a^{j_1}\cdot b^{i_2}a^{j_2})&=
  \left\{
    \begin{array}{ll}
      \alpha(b^{i_1-j_1+i_2}a^{j_2}), & \hbox{if~} j_1\leqslant i_2;\\
      \alpha(b^{i_1}a^{j_1-i_2+j_2}), & \hbox{if~} j_1> i_2.
    \end{array}
  \right.= \\
   &=
  \left\{
    \begin{array}{ll}
      b^{j_2}a^{i_1-j_1+i_2}, & \hbox{if~} j_1<i_2;\\
      b^{j_2}a^{i_1}, & \hbox{if~} j_1=i_2;\\
      b^{j_1-i_2+j_2}a^{i_1}, & \hbox{if~} j_1> i_2
    \end{array}
  \right.
\end{align*}
and
\begin{align*}
  \alpha(b^{i_1}a^{j_1})\cdot \alpha(b^{i_2}a^{j_2})&=b^{j_1}a^{i_1}\cdot b^{j_2}a^{i_2}= \\
   &=
  \left\{
    \begin{array}{ll}
      b^{j_2}a^{i_1-j_1+i_2}, & \hbox{if~} j_1<i_2;\\
      b^{j_2}a^{i_1}, & \hbox{if~} j_1=i_2;\\
      b^{j_1-i_2+j_2}a^{i_1}, & \hbox{if~} j_1> i_2.
    \end{array}
  \right.
\end{align*}
This implies that the map $\alpha$ is an anti-homomorphism. It is obvious that the so defined map $\alpha\colon \mathscr{C}_{+}(a,b)\rightarrow\mathscr{C}_{-}(a,b)$ is bijective, and hence it is an anti-isomorphism of monoids $\mathscr{C}_{+}(a,b)$ and $\mathscr{C}_{-}(a,b)$, because $\alpha(b^0a^0)=b^0a^0=1$.
\end{proof}

If $S$ is a semigroup, then we shall denote the Green relations on $S$ by $\mathscr{R}$, $\mathscr{L}$, $\mathscr{J}$, $\mathscr{D}$ and $\mathscr{H}$ (see \cite[Section~2.1]{Clifford-Preston=1961}):
\begin{align*}
    &\qquad a\mathscr{R}b \mbox{ if and only if } aS^1=bS^1;\\
    &\qquad a\mathscr{L}b \mbox{ if and only if } S^1a=S^1b;\\
    &\qquad a\mathscr{J}b \mbox{ if and only if } S^1aS^1=S^1bS^1;\\
    &\qquad \mathscr{D}=\mathscr{L}\circ\mathscr{R}=
          \mathscr{R}\circ\mathscr{L};\\
    &\qquad \mathscr{H}=\mathscr{L}\cap\mathscr{R}.
\end{align*}

The following proposition describes Green's relations on  monoids $\mathscr{C}_{+}(a,b)$ and $\mathscr{C}_{-}(a,b)$.

\begin{proposition}
Green's relations $\mathscr{R}$, $\mathscr{L}$, $\mathscr{J}$, $\mathscr{D}$ and $\mathscr{H}$ on  monoids $\mathscr{C}_{+}(a,b)$ and $\mathscr{C}_{-}(a,b)$  coincide with the equality relation.
\end{proposition}

\begin{proof}
Suppose that $b^ia^j\mathscr{L}b^ka^l$ in $\mathscr{C}_{+}(a,b)$ for some $i,j,k,l\in\omega$, $i\leqslant j$, and $k\leqslant l$. Then there exist $b^xa^y,b^ua^v\in \mathscr{C}_{+}(a,b)$ such that $b^ia^j=b^ua^v\cdot b^ka^l$ and $b^ka^l=b^xa^y\cdot b^ia^j$. Then the equalities
\begin{equation}\label{eq-2.1}
  b^ka^l=b^xa^y\cdot b^ia^j=
  \left\{
    \begin{array}{ll}
      b^{x-y+i}a^{j}, & \hbox{if~} y<i;\\
      b^{x}a^{j},     & \hbox{if~} y=i; \\
      b^{x}a^{y-i+j}, & \hbox{if~} y>i
    \end{array}
  \right.
\end{equation}
imply that $l\geqslant j$, and by the equalities
\begin{equation}\label{eq-2.2}
  b^ia^j=b^ua^v\cdot b^ka^l=
  \left\{
    \begin{array}{ll}
      b^{u-v+k}a^{l}, & \hbox{if~} v<k;\\
      b^{u}a^{l},     & \hbox{if~} v=k; \\
      b^{u}a^{k-v+l}, & \hbox{if~} v>k
    \end{array}
  \right.
\end{equation}
we get that $l\leqslant j$, and hence $l=j$. Also, the equalities  \eqref{eq-2.1} imply that $k-l=(x-y)+(i-j)$. Since $x\leqslant y$, we have that  $k-l\leqslant i-j$. Similar, the equalities  \eqref{eq-2.2} imply that $i-j=(u-v)+(k-l)$, and hence $i-j\leqslant k-l$. Thus, $i-j=k-l$. Since $l=j$, we obtain that $k=i$. Hence $\mathscr{L}$ is the equality relation on the monoid $\mathscr{C}_{+}(a,b)$.

\smallskip

Suppose that $b^ia^j\mathscr{R}b^ka^l$ in $\mathscr{C}_{+}(a,b)$ for some $i,j,k,l\in\omega$, $i\leqslant j$, and $k\leqslant l$. Then there exist $b^xa^y,b^ua^v\in \mathscr{C}_{+}(a,b)$ such that $b^ia^j=b^ka^l\cdot b^ua^v$ and $b^ka^l=b^ia^j \cdot b^xa^y$. By the equalities
\begin{equation}\label{eq-2.3}
  b^ka^l=b^ia^j \cdot b^xa^y=
  \left\{
    \begin{array}{ll}
      b^{i-j+x}a^{y}, & \hbox{if~} j<x;\\
      b^{i}a^{y},     & \hbox{if~} j=x; \\
      b^{i}a^{j-x+y}, & \hbox{if~} j>x
    \end{array}
  \right.
\end{equation}
we get that $k\geqslant i$, and by the equalities
\begin{equation}\label{eq-2.4}
  b^ia^j=b^ka^l\cdot b^ua^v=
  \left\{
    \begin{array}{ll}
      b^{k-l+u}a^{v}, & \hbox{if~} l<u;\\
      b^{k}a^{v},     & \hbox{if~} l=u; \\
      b^{k}a^{l-u+v}, & \hbox{if~} l>u
    \end{array}
  \right.
\end{equation}
we get that $k\leqslant i$, and hence $k=i$. The equalities  \eqref{eq-2.3} imply that $k-l=(i-j)+(x-y)$. Since $x\leqslant y$, we have that  $k-l\geqslant i-j$. Similar, the equalities  \eqref{eq-2.4} imply that $i-j=(k-l)+(u-v)$, and hence $i-j\geqslant k-l$. Thus, $i-j=k-l$. Since $l=j$, we obtain that $k=i$. Hence $\mathscr{R}$ is the equality relation on the monoid $\mathscr{C}_{+}(a,b)$.

\smallskip

Since $\mathscr{H}=\mathscr{L}\cap\mathscr{R}$ and $\mathscr{D}=\mathscr{L}\circ\mathscr{R}= \mathscr{R}\circ\mathscr{L}$, the previous part of the proof imply that $\mathscr{H}=\mathscr{D}=\mathscr{L}=\mathscr{R}$ in $\mathscr{C}_{+}(a,b)$.

\smallskip

Suppose that $b^ia^j\mathscr{J}b^ka^l$ in $\mathscr{C}_{+}(a,b)$ for some $i,j,k,l\in\omega$, $i\leqslant j$, and $k\leqslant l$. Then there exist $b^{k_1}a^{l_1},b^{k_2}a^{l_2},b^{i_1}a^{j_1},b^{i_2}a^{j_2}\in \mathscr{C}_{+}(a,b)$ such that
\begin{equation}\label{eq-2.5}
  b^ia^j=b^{k_1}a^{l_1}\cdot b^ka^l\cdot b^{k_2}a^{l_2}
\end{equation}
and
\begin{equation}\label{eq-2.6}
  b^ka^l=b^{i_1}a^{j_1}\cdot b^ia^j\cdot b^{i_2}a^{j_2}.
\end{equation}
The semigroup operation of $\mathscr{C}_{+}(a,b)$ implies that
\begin{equation}\label{eq-2.7}
  i-j=(k_1-l_1)+(k-l)+(k_2-l_2),
\end{equation}
and since $k_1\leqslant l_1$ and $k_2\leqslant l_2$, we get that $i-j\leqslant k-l$. Similar we get that
\begin{equation}\label{eq-2.8}
  k-l=(i_1-j_1)+(i-j)+(i_2-j_2),
\end{equation}
and since $i_1\leqslant j_1$ and $i_2\leqslant j_2$, we have that $k-l\leqslant i-j$. Hence we obtain that $i-j=k-l$. The last equality and equalities \eqref{eq-2.7} and \eqref{eq-2.8} imply that
\begin{equation*}
  i_1=j_1, \quad i_2=j_2, \quad k_1=l_1, \quad \hbox{and} \quad k_2=l_2,
\end{equation*}
because $i_1\leqslant j_1$, $i_2\leqslant j_2$, $k_1\leqslant l_1$, and $k_2\leqslant l_2$. Then the semigroup operation of $\mathscr{C}_{+}(a,b)$ implies that
\begin{equation*}
  i_1=j_1\leqslant k, \quad i_2=j_2\leqslant l, \quad k_1=l_1\leqslant i, \quad \hbox{and} \quad k_2=l_2\leqslant j.
\end{equation*}

Then we have that
\begin{equation*}
  b^ia^j=b^{k_1}a^{k_1}\cdot b^ka^l\cdot b^{k_2}a^{k_2}= b^{k_1}a^{k_1}\cdot b^{i_1}a^{i_1}\cdot b^ia^j\cdot b^{i_2}a^{i_2}\cdot b^{k_2}a^{k_2}
\end{equation*}
and
\begin{equation*}
  b^ka^l=b^{i_1}a^{i_1}\cdot b^ia^j\cdot b^{i_2}a^{i_2}=b^{i_1}a^{i_1}\cdot b^{k_1}a^{k_1}\cdot b^ka^l\cdot b^{k_2}a^{k_2}\cdot b^{i_2}a^{i_2}.
\end{equation*}
Since idempotents commute in $\mathscr{C}_{+}(a,b)$, the last two equalities and the semigroup operation of $\mathscr{C}_{+}(a,b)$ imply that $\max\{k_1,i_1\}\leqslant i,k$ and $\max\{k_2,i_2\}\leqslant j,l$. Hence, again using the semigroup operation of $\mathscr{C}_{+}(a,b)$ we get that $ b^ia^j=b^ka^l$. Hence $\mathscr{J}$ is the equality relation on the monoid $\mathscr{C}_{+}(a,b)$.

\smallskip

Appling Proposition~\ref{proposition-4} we obtain that the statement of the proposition holds for the monoid $\mathscr{C}_{-}(a,b)$.
\end{proof}

Since
\begin{equation*}
  b^xa^y\cdot b^ia^i=
  \left\{
    \begin{array}{ll}
      b^{x-y+i}a^{i}, & \hbox{if~} y<i;\\
      b^{x}a^{i},     & \hbox{if~} y=i; \\
      b^{x}a^{y},     & \hbox{if~} y>i,
    \end{array}
  \right.
\end{equation*}
we have that
\begin{equation}\label{eq-2.9}
  \mathscr{C}_{+}(a,b)\cdot b^ia^i=
  \left\{b^sa^t\in\mathscr{C}_{+}(a,b)\colon t\geqslant i\right\}
\end{equation}
for any $i\in\omega$. This implies that if $\tau$ is a Hausdorff left-continuous topology on the semigroup $\mathscr{C}_{+}(a,b)$ then for any idempotent $b^ia^i\in \mathscr{C}_{+}(a,b)$ the right shift $\rho_{b^ia^i}\colon \mathscr{C}_{+}(a,b)\to \mathscr{C}_{+}(a,b)$ is a retraction, $\mathscr{C}_{+}(a,b)\cdot b^ia^i$ is a retract of $(\mathscr{C}_{+}(a,b),\tau)$, and hence $\mathscr{C}_{+}(a,b)\cdot b^ia^i$ is a closed subset of the topological space $(\mathscr{C}_{+}(a,b),\tau)$ (see \cite[Ex.~1.5.C]{Engelking=1989}). The above arguments imply that every element $b^ka^l$ of the monoid $\mathscr{C}_{+}(a,b)$ has a finite open neighbourhood in the space $(\mathscr{C}_{+}(a,b),\tau)$. Since the topology $\tau$ is Hausdorff, $b^la^l$ is an isolated point $(\mathscr{C}_{+}(a,b),\tau)$. Hence we proved the following theorem.

\begin{theorem}\label{theorem-6}
Every Hausdorff left-continuous topology on the monoid $\mathscr{C}_{+}(a,b)$  is discrete.
\end{theorem}

Proposition~\ref{proposition-4} and Theorem~\ref{theorem-6} imply the following theorem.

\begin{theorem}\label{theorem-7}
Every Hausdorff right-continuous topology on the monoid $\mathscr{C}_{-}(a,b)$  is discrete.
\end{theorem}

In the paper \cite{Chornenka-Gutik=2023} the following two examples are constructed.

\begin{example}[{\cite[Example~2]{Chornenka-Gutik=2023}}]\label{example-8}
The topology $\tau_2$ on the bicyclic monoid ${\mathscr{C}}(a,b)$ is defined in the following way. For any $b^ia^j\in{\mathscr{C}}(a,b)$ and any non-negahtive integer $n$ put
\begin{equation*}
  O_n(b^ia^j)=\left\{b^ia^j\right\}\cup \{b^{i+l}a^{j+l}\colon l>n\}.
\end{equation*}
Let $\mathscr{B}_2(b^ia^j)=\left\{O_n(b^ia^j)\colon n\in\omega\right\}$ be the system of open neighbourhoods at the point $b^ia^j\in{\mathscr{C}}(a,b)$. It is obvious that the family $\mathscr{B}_2=\bigcup_{i,j\in\omega}\mathscr{B}_2(b^ia^j)$ satisfies the properties (BP1)--(BP3) of \cite{Engelking=1989}, and hence it generates a topology on ${\mathscr{C}}(a,b)$.
\end{example}

\begin{example}[{\cite[Example~3]{Chornenka-Gutik=2023}}]\label{example-9}
The topology $\tau_{\mathrm{c}}$ on the bicyclic semigroup  ${\mathscr{C}}(a,b)$ is defined in the following way. For any non-negative integer $n$ put
\begin{equation*}
  C_n=\left\{b^ia^j\in{\mathscr{C}}(a,b)\colon i,j\leqslant n\right\}.
\end{equation*}
Let
\begin{equation*}
\mathscr{B}_{\mathrm{c}}(b^ia^j)=\left\{W_n(b^ia^j)=\left\{b^ia^j\right\}\cup{\mathscr{C}}(a,b)\setminus C_n\colon n\in\omega\right\}
\end{equation*}
be the system of open neighbourhoods at the point $b^ia^j\in{\mathscr{C}}(a,b)$. It is obvious that the family $\mathscr{B}_{\mathrm{c}}=\displaystyle\bigcup_{i,j\in\omega}\mathscr{B}_{\mathrm{c}}(b^ia^j)$ satisfies the properties (BP1)--(BP3) of \cite{Engelking=1989}, and hence it generates the topology $\tau_{\mathrm{c}}$ on ${\mathscr{C}}(a,b)$.
\end{example}

By Proposition~2 of \cite{Chornenka-Gutik=2023}, $\tau_2$ is a locally compact semigroup $T_1$-topology on the bicyclic semigroup  ${\mathscr{C}}(a,b)$. Simple verifications show that $\tau_2$ induces on the monoid $\mathscr{C}_{+}(a,b)$  a locally compact semigroup $T_1$-topology. Also, by Proposition~3 of \cite{Chornenka-Gutik=2023}, $\tau_{\mathrm{c}}$ is a shift-continuous compact  $T_1$-topology  on ${\mathscr{C}}(a,b)$. It is obvious that $\tau_{\mathrm{c}}$ induces on the monoid $\mathscr{C}_{+}(a,b)$  a shift-continuous compact $T_1$-topology.

\smallskip

Lemma~I.1 of \cite{Eberhart-Selden=1969} implies

\begin{lemma}\label{lemma-10}
For each $v,w\in \mathscr{C}_{+}(a,b)$ both sets $\{u\in \mathscr{C}_{+}(a,b)) \colon vu=w\}$ and $\{u\in \mathscr{C}_{+}(a,b) \colon uv=w\}$ are finite.
\end{lemma}

Proposition~\ref{proposition-11} describes the closure of the monoid $\mathscr{C}_{+}(a,b)$ in a Hausdorff semitopological monoid.

\begin{proposition}\label{proposition-11}
If the monoid $\mathscr{C}_{+}(a,b)$ is a dense subsemigroup of a Hausdorff semitopological monoid $S$ and $I=S\setminus\mathscr{C}_{+}(a,b)\neq\varnothing$ then $I$ is a closed two-sided ideal of the semigroup $S$.
\end{proposition}

\begin{proof}
Since every discrete space is locally compact, Theorem~3.3.9 of \cite{Engelking=1989} implies that $\mathscr{C}_{+}(a,b)$ is an open subset of $S$.

\smallskip

Fix an arbitrary element $y\in I$. If $xy=z\notin I$ for some $x\in\mathscr{C}_{+}(a,b)$ then there exists an open neighbourhood $U(y)$ of the point $y$ in the space $S$ such that $\{x\}\cdot U(y)=\{z\}\subset\mathscr{C}_{+}(a,b)$. The neighbourhood $U(y)$ contains infinitely many elements of the semigroup $\mathscr{C}_{+}(a,b)$. This contradicts Lemma~\ref{lemma-10}. The obtained contradiction implies that $xy\in I$ for all $x\in  \mathscr{C}_{+}(a,b)$ and $y\in I$. The proof of the statement that $yx\in I$ for all $x\in  \mathscr{C}_{+}(a,b)$ and $y\in I$ is similar.

\smallskip

Suppose to the contrary that $xy=w\notin I$ for some $x,y\in I$. Then $w\in \mathscr{C}_{+}(a,b)$ and the separate continuity of the semigroup operation in $S$ implies that there exist open neighbourhoods $U(x)$ and $U(y)$ of the points $x$ and $y$ in $S$, respectively, such that $\{x\}\cdot U(y)=\{w\}$ and $U(x)\cdot \{y\}=\{w\}$. Since both neighbourhoods $U(x)$ and $U(y)$ contain infinitely many elements of the semigroup $\mathscr{C}_{+}(a,b)$, both equalities $\{x\}\cdot U(y)=\{w\}$ and $U(x)\cdot \{y\}=\{w\}$ contradict mentioned above Lemma~\ref{lemma-10}. The obtained contradiction implies that $xy\in I$.
\end{proof}

The proof of Proposition~\ref{proposition-12}  is similar to Proposition~\ref{proposition-11}.

\begin{proposition}\label{proposition-12}
If the monoid $\mathscr{C}_{-}(a,b)$ is a dense subsemigroup of a Hausdorff semitopological monoid $S$ and $I=S\setminus\mathscr{C}_{-}(a,b)\neq\varnothing$ then $I$ is a closed two-sided ideal of the semigroup $S$.
\end{proposition}

It is well known that neither the bicyclic monoid no the Taimanov semogroup $\mathfrak{T}$ do not embed into compact Hausdorff topological semigroups \cite{Gutik=2016, Koch-Wallace=1957}. Later we show that there exists a compact Hausdorff topological monoid $S$ which contains the monoid $\mathscr{C}_{+}(a,b)$ as a dense submonoid.

\begin{example}
Put $S=\mathscr{C}_{+}(a,b)\sqcup\{0\}$ is the monoid $\mathscr{C}_{+}(a,b)$ with the adjoined zero $0$. We define the topology $\tau_S$ on the semigroup $S$ in the following way. All points of $\mathscr{C}_{+}(a,b)$ are isolated in $(S,\tau_S)$ and put the family $\mathscr{B}_S=\left\{U_p(0)\colon p\in\omega\right\}$, where \begin{equation*}
  U_p(0)=\{0\}\cup\left\{b^{i}a^{j}\in \mathscr{C}_{+}(a,b)\colon j\geqslant p\right\},
\end{equation*}
is the system of open neighbourhoods at zero $0$. It is obvious that the family $\mathscr{B}_{S}$ satisfies the properties (BP1)--(BP3) of \cite{Engelking=1989}, and hence it generates the topology $\tau_{S}$ on the monoid $S$. It is obvious that $(S,\tau_S)$ is a Hausdorff topological space, and since the set $U_p(0)\setminus\{0\}$ has the finite complement in $\mathscr{C}_{+}(a,b)$, the space $(S,\tau_S)$ is compact.

The following equality
\begin{equation}\label{eq-2.10}
  b^xa^y\cdot b^ia^i=
  \left\{
    \begin{array}{ll}
      b^{x-y+i}a^{i}, & \hbox{if~} y<i;\\
      b^{x}a^{i},     & \hbox{if~} y=i; \\
      b^{x}a^{y},     & \hbox{if~} y>i,
    \end{array}
  \right.
\end{equation}
implies that $U_p(0)\cdot U_p(0)\subseteq U_p(0)$ and $b^ka^l\cdot U_p(0)\subseteq U_p(0)$ for any $b^ka^l\in \mathscr{C}_{+}(a,b)$. Also by \eqref{eq-2.10} we have that $U_p(0)\cdot b^ia^i\subseteq U_p(0)$ for $p\geqslant i$. Hence $(S,\tau_S)$ is a topological semigroup.
\end{example}

\section{\textbf{Some examples}}

In this section we construct non-discrete Hausdorff right-continuous topology on the monoid $\mathscr{C}_{+}(a,b)$ which is not left-continuous.

\smallskip

By $(\omega,+)$ we denote the additive semigroup of non-negative integers.

\smallskip

The semigroup operation of $\mathscr{C}_{+}(a,b)$ implies that
\begin{equation*}
  R_k=\left\{b^ka^{k+s}\in\mathscr{C}_{+}(a,b)\colon s\in\omega \right\}
\end{equation*}
is a subsemigroup of $\mathscr{C}_{+}(a,b)$ for any $k\in\omega$. For any $k\in\omega$ we define the map $\iota_k\colon (\omega,+)\to R_k$ by the formula $\iota_k(s)=b^ka^{k+s}$. Since $\iota_k(s_1+s_2)=b^ka^{k+s_1+s_2}$ and
\begin{equation*}
  \iota_k(s_1)\cdot \iota_k(s_2)=b^ka^{k+s_1}\cdot b^ka^{k+s_2}=b^ka^{s_1}\cdot a^{k+s_2}=b^ka^{k+s_1+s_2},
\end{equation*}
we obtain that the map $\iota_k$ is a monoid homomorphism. It is obvious that the map $\iota_k$ is bijective, and hence it is an isomorphism.

\smallskip

Let $p$ be a prime positive integer. Then the family of subgroups $\{p^n\mathbb{Z}\colon n\in \mathbb{N}\}$ of the additive group of integers $\mathbb{Z}$ form a fundamental system of neighborhoods of the zero $0$ for a linear precompact topology $\tau_p$ on $\mathbb{Z}$, which is usually called the \emph{$p$-adic topology} (see \cite[p.~45]{Dikranjan-Stoyanov-Prodanov=1990}). It is well known that for any prime integer $p$ the $p$-adic topology $\tau_p$ on the additive group of integers $\mathbb{Z}$ is a group topology, i.e., the group operation and the inversion are continuous in $(\mathbb{Z},\tau_p)$.

\begin{example}
Fix an arbitrary prime positive integer $p$. Put $\tau_p^i$ is the induced topology on $(\omega,+)$ from $(\mathbb{Z},\tau_p)$. Since $\tau_p$ is a group topology on the additive group of integers $\mathbb{Z}$ and $(\omega,+)$ is a subsemigroup of $\mathbb{Z}$, we have that $\tau_p^i$ is a semigroup topology on $(\omega,+)$. Since the $T_0$-space of a topological group is completely regular (Tychonoff), Theorem~2.1.6 of \cite{Engelking=1989} implies that $(\omega,+,\tau_p^i)$ is a completely regular space. Also, by Hausdorffness of $(\omega,\tau_p^i)$ we have that the family $\mathscr{B}_p^i(s)=\left\{U_n(s)\colon n\in\omega\right\}$, where $U_n(s)=\left\{s+p^nj\colon j\in\omega\right\}$, determines the system of open neighbourhoods at the point $s$ in the space $(\omega,\tau_p^i)$.

\smallskip

We define the topology $\tau_p^+$ on the semigroup $\mathscr{C}_{+}(a,b)$ in the following way. For any $k,s\in\omega$ we denote $W_n(b^ka^{k+s})=\iota_k(U_n(s))$ and  put the family
\begin{equation*}
  \mathscr{B}_p^+(b^ka^{k+s})=\left\{W_n(b^ka^{k+s})\colon n\in\omega\right\}
\end{equation*}
is the system of open neighbourhoods at the point $b^ka^{k+s}$. It is obvious that the family $\mathscr{B}_p^+=\left\{\mathscr{B}_p^+(b^ka^{k+s})\colon k,s\in\omega\right\}$ satisfies the properties (BP1)--(BP3) of \cite{Engelking=1989}, and hence it generates the topology $\tau_p^+$ on the monoid $\mathscr{C}_{+}(a,b)$.

\smallskip

Simple calculations show that
\begin{equation*}
  W_n(b^ka^{k+s})=\left\{b^ka^{k+s+p^nj}\colon j\in\omega\right\}
\end{equation*}
for any $k,s\in \omega$. This implies that the topological space $(\mathscr{C}_{+}(a,b),\tau_p^+)$ is homeomorphic to a countable topological sum of spaces $(\omega,\tau_p^i)$. By the Birkhoff-Kakutani Theorem (see \cite[Section 1.22]{Montgomery-Zippin=1955}) the topological group $(\mathbb{Z},\tau_p)$ is metrizable, and hence by Theorem~4.2.1 of \cite{Engelking=1989} the space $(\mathscr{C}_{+}(a,b),\tau_p^+)$ is metrizable as well. Also by Corollary~4.1.13 of \cite{Engelking=1989} the space $(\mathscr{C}_{+}(a,b),\tau_p^+)$ is perfectly normal, i.e., $(\mathscr{C}_{+}(a,b),\tau_p^+)$  is a normal space and every
closed subset of $(\mathscr{C}_{+}(a,b),\tau_p^+)$ is a $G_\delta$-set.
\end{example}

\begin{proposition}\label{proposition-15}
$(\mathscr{C}_{+}(a,b),\tau_p^+)$ is a right topological semigroup. Moreover $(\mathscr{C}_{+}(a,b),\tau_p^+)$ is not a left topological semigroup.
\end{proposition}

\begin{proof}
Fix arbitrary $b^{k_1}a^{k_1+s_1}, b^{k_2}a^{k_2+s_2} \in\mathscr{C}_{+}(a,b)$, $k_1,k_2,s_1,s_2\in\omega$.

\smallskip

We consider the possible cases.

\smallskip

If $k_1+s_1<k_2$, then
\begin{equation*}
b^{k_1}a^{k_1+s_1}\cdot b^{k_2}a^{k_2+s_2}= b^{k_1-k_1-s_1+k_2}a^{k_2+s_2}=b^{k_2-s_1}a^{k_2+s_2},
\end{equation*}
and for any $n\in\omega$ we get that
\begin{align*}
  b^{k_1}a^{k_1+s_1}\cdot W_n(b^{k_2}a^{k_2+s_2})&=\left\{b^{k_1}a^{k_1+s_1}\cdot b^{k_2}a^{k_2+s_2+p^nj}\colon j\in\omega\right\}= \\
   &=\left\{b^{k_2-s_1}a^{k_2+s_2+p^nj}\colon j\in\omega\right\}= \\
   &=W_n(b^{k_2-s_1}a^{k_2+s_2}).
\end{align*}

If $k_1+s_1=k_2$, then
\begin{equation*}
b^{k_1}a^{k_1+s_1}\cdot b^{k_2}a^{k_2+s_2}= b^{k_1}a^{k_2+s_2},
\end{equation*}
and for any $n\in\omega$ we obtain that
\begin{align*}
  b^{k_1}a^{k_1+s_1}\cdot W_n(b^{k_2}a^{k_2+s_2})&=\left\{b^{k_1}a^{k_1+s_1}\cdot b^{k_2}a^{k_2+s_2+p^nj}\colon j\in\omega\right\}= \\
   &=\left\{b^{k_1}a^{k_2+s_2+p^nj}\colon j\in\omega\right\}= \\
   &=W_n(b^{k_1}a^{k_2+s_2}).
\end{align*}

If $k_1+s_1>k_2$, then
\begin{equation*}
b^{k_1}a^{k_1+s_1}\cdot b^{k_2}a^{k_2+s_2}= b^{k_1}a^{k_1+s_1-k_2+k_2+s_2}=b^{k_1}a^{k_1+s_1+s_2},
\end{equation*}
and for any $n\in\omega$ we get that
\begin{align*}
  b^{k_1}a^{k_1+s_1}\cdot W_n(b^{k_2}a^{k_2+s_2})&=\left\{b^{k_1}a^{k_1+s_1}\cdot b^{k_2}a^{k_2+s_2+p^nj}\colon j\in\omega\right\}= \\
   &=\left\{b^{k_1}a^{k_1+s_1+s_2+p^nj}\colon j\in\omega\right\}= \\
   &=W_n(b^{k_1}a^{k_1+s_1+s_2}).
\end{align*}

The above arguments imply that $(\mathscr{C}_{+}(a,b),\tau_p^+)$ is a right topological semigroup.

\smallskip

Next we show that the second statement holds. It is obvious that $1\cdot ba=ba$. For any open neighbourhood $W_n(1)=\left\{a^{p^nj}\colon j\in\omega\right\}$ of the unit element $1$ in $(\mathscr{C}_{+}(a,b),\tau_p^+)$ we have that
\begin{equation*}
  W_n(1)\cdot ba=\big\{a^{p^nj}\cdot ba\colon j\in\omega\big\}.
\end{equation*}
Then for any positive integer $j$ we get that
\begin{equation*}
  a^{p^nj}\cdot ba=a^{p^nj}\notin W_m(ba)
\end{equation*}
for any $m\in\omega$. This completes the second part of the proposition.
\end{proof}

\begin{remark}\label{remark-16}
Simple verifications show that for arbitrary $b^{k_1}a^{k_1+s_1}, b^{k_2}a^{k_2+s_2} \in\mathscr{C}_{+}(a,b)$, $k_1,k_2,s_1,s_2\in\omega$ with $k_1+s_1\leqslant k_2$, the semigroup operation is not right-continuous in $(\mathscr{C}_{+}(a,b),\tau_p^+)$. The proof of this statement  is similar to the second statement of Proposition~\ref{proposition-15}.
\end{remark}

Propositions~\ref{proposition-4} and \ref{proposition-15} imply the following.

\begin{proposition}\label{proposition-17}
The semigroup $\mathscr{C}_{-}(a,b)$ admits a non-discrete left-continuous topology $\tau_p^-$ which is not right-continuous.
\end{proposition}

\begin{remark}\label{remark-18}
Alex Ravsky in Topological Algebra Seminar at Lviv University posed the following question: \emph{Is any Hausdorff left-continuous (right-continuous) topology on the semigroup $\mathscr{C}_{+}(a,b)$ ($\mathscr{C}_{+}(a,b)$) discrete?}
Propositions~\ref{proposition-15} and \ref{proposition-17} give negative answer on this question.
\end{remark}

%%%%%%%%%%%%%%%%%%%%%%%%%%%%%%%%%%%%%%%%%%%%%%%%%%%%%
\section*{\textbf{Acknowledgements}}

The authors acknowledge Alex Ravsky and the referee for his/her comments and suggestions.

\end{document}